\documentclass[a4paper,12pt,leqno]{amsart}
\usepackage{amsmath,amssymb,mathrsfs}
\usepackage{xcolor,hyperref}
\hypersetup{colorlinks=true,citecolor=black,linkcolor=black}
\newcommand{\h}{\hbox}
\newcommand{\q}{\quad}
\newcommand{\nin}{\noindent}

\newcommand{\ms}{\par\medskip}
\newcommand{\sk}{\par\smallskip}

\newcommand{\msn}{\par\medskip\noindent}
\newcommand{\skn}{\par\smallskip\noindent}
\newcommand{\ges}{\geqslant}
\newcommand{\les}{\leqslant}
\newcommand{\one}{\hskip1pt}
\newcommand{\mcap}{\hbox{$\bigcap$}}
\newcommand{\mcup}{\hbox{$\bigcup$}}
\newcommand{\msum}{\hbox{$\sum$}}

\newcommand{\D}{{\mathcal D}}
\newcommand{\OO}{{\mathcal O}}

\newcommand{\I}{{\mathcal I}}
\newcommand{\Sc}{{\mathcal S}}
\newcommand{\X}{{\mathcal X}}
\newcommand{\R}{{\mathcal R}}
\newcommand{\PP}{{\mathbb P}}
\newcommand{\Q}{{\mathbb Q}}
\newcommand{\C}{{\mathbb C}}

\newcommand{\N}{{\mathbb N}}

\newcommand{\Z}{{\mathbb Z}}
\newcommand{\ee}{{\mathbf e}}
\newcommand{\alt}{\widetilde{\alpha}}
\newcommand{\Ht}{\widetilde{H}}
\newcommand{\dfw}{{\rm d}f{\wedge}}
\newcommand{\df}{{\rm d}f}
\newcommand{\Ff}{F_{\!f}}
\newcommand{\kod}{\tfrac{k}{d}}
\newcommand{\Gr}{{\rm Gr}}

\newcommand{\al}{\alpha}
\newcommand{\be}{\beta}
\newcommand{\ga}{\gamma}

\newcommand{\de}{\delta}

\newcommand{\la}{\lambda}

\newcommand{\Si}{\Sigma}

\newcommand{\Om}{\Omega}
\newcommand{\mm}{\mathfrak m}
\newcommand{\dd}{\partial}
\newcommand{\ddd}{{\rm d}}
\newcommand{\tos}{\,{\to}\,}
\newcommand{\eq}{\,{=}\,}
\newcommand{\defs}{\,{:=}\,}
\newcommand{\nes}{\,{\ne}\,}
\newcommand{\ins}{\,{\in}\,}
\newcommand{\sst}{\,{\subset}\,}
\newcommand{\stm}{\,{\setminus}\,}
\newcommand{\gess}{\,{\ges}\,}
\newcommand{\less}{\,{\les}\,}
\newcommand{\sgt}{\,{>}\,}
\newcommand{\slt}{\,{<}\,}
\newcommand{\col}{\,{:}\,}
\newcommand{\pl}{\one {+}\one}
\newcommand{\mi}{\one {-}\one}
\newcommand{\bl}{\bigl}
\newcommand{\br}{\bigr}

\newcommand{\ssb}{\raise.15ex\h{${\scriptscriptstyle\bullet}$}}
\newcommand{\ssc}{\,\raise.15ex\h{${\scriptstyle\circ}$}\,}

\newcommand{\into}{\hookrightarrow}
\newcommand{\simto}{\,\,\rlap{\hskip1.5mm\raise1.4mm\hbox{$\sim$}}\hbox{$\longrightarrow$}\,\,}

\makeatletter
\renewcommand\section{\@startsection{section}{1}{0pt}{-3ex plus -1ex minus -.2ex}{2.3ex plus.2ex}{\centering\normalfont\bfseries}}
\makeatother
\theoremstyle{plain}
\newtheorem{thm}{Theorem}[section]

\newtheorem{prop}[thm]{Proposition}
\newtheorem{lem}[thm]{Lemma}
\newtheorem{ithm}{Theorem}

\newtheorem{iprop}{Proposition}

\theoremstyle{definition}
\newtheorem{rem}[thm]{Remark}
\newtheorem{conj}[thm]{Conjecture}
\newtheorem{exam}[thm]{Example}
\newtheorem{defi}[thm]{Definition}
\newtheorem{ques}[thm]{Question}

\newtheorem{irem}{Remark}

\newtheorem{iques}{Question}

\begin{document}
\title[Roots of Bernstein-Sato polynomials]{Roots of Bernstein-Sato polynomials for projective hypersurfaces with ordinary double points}
\author[S.-J. Jung]{Seung-Jo Jung}
\address{S.-J. Jung : Department of Mathematics Education, and Institute of Pure and Applied Mathematics, Jeonbuk National University, Jeonju, 54896, Korea}
\email{seungjo@jbnu.ac.kr}
\author[M. Saito]{Morihiko Saito}
\address{M. Saito : RIMS Kyoto University, Kyoto 606-8502 Japan}
\email{msaito@kurims.kyoto-u.ac.jp}
\thanks{This work was partially supported by National Research Foundation of Korea RS-2026-25472637.}
\begin{abstract}
Let $X\subset{\mathbb P}^{n-1}$ be a hypersurface of degree $d\ge3$ with ordinary double points, where $n\ge3$. The roots of Bernstein-Sato polynomial of its defining polynomial $f$ are given up to sign by 1, $(n-1)/2$, and $j/d$ for $j\in{\mathbb Z}\cap[n,nd-n-p_f]$ with $p_f$ a positive integer. Here $p_f$ is bounded above by the minimal positive integer $q_s$ satisfying $\binom{q_s+n-1}{n-1}>s:=|{\rm Sing}\,X|$, and we can verify that $p_f$ coincides with $q_s$ in the case the singular points of $X$ are in ``general position". We show that this upper bound is sharp in the case $\binom{\lfloor d/2\rfloor+n-2}{n-1}\ge s$ or $\binom{d+n-3}{n-1}\ge sn$ by providing a homogeneous polynomial of degree $d$ such that the associated projective hypersurface has ordinary double points at given $s$ points in sufficiently general position and is nonsingular outside them (using a theorem of Alexander and Hirschowitz for the second case). It is conjectured that the above sharp bound under the first hypothesis can be extended naturally to the case where $X$ has only $A_2$-singularities instead of ordinary double points.
\end{abstract}
\maketitle

\section*{Introduction} \label{intr}
\nin
Let $X\sst\PP^{n-1}$ be a projective hypersurface of degree $d\gess3$ having only ordinary double points as singularities, where $n\gess3$. Let $\R_f\sst\Q_{>0}$ be the set of roots of $b_f(s)$ up to sign with $b_f(s)$ the Bernstein-Sato polynomial of a defining polynomial $f$ of $X$. We have the following.

\begin{ithm}[{\cite[Theorem 5]{wh}}] \label{T1}
If the singularities of $X$ are ordinary double points, then
\begin{equation} \label{1}
\R_f=\tfrac{1}{d}\one\bl(\Z\cap[n,nd\mi n\mi p_f]\br)\cup\R_X,
\end{equation}
where $p_f$ is a positive integer and $\R_X\eq\bl\{1,\tfrac{n-1}{2}\br\}$, the roots of Bernstein-Sato polynomial of ordinary double points of $X$.
\end{ithm}

This follows from the $E_2$-degeneration of the pole order spectral sequence \cite[Theorem 2]{wh} and \cite[Theorem 2]{bcm} (see also Theorems~\ref{T1.1} and \ref{T2.1} below) using the symmetries of the $E_1$-term \cite[Corollaries 1 and 2]{kosz} (see also Theorem~\ref{T3.1} below) together with some vanishing theorems in \cite{DiSt}, \cite{Di} (see also \cite{odp} and Theorem~\ref{T3.3} below). Note that the assumption $d\gess n$ in \cite[Theorem 5]{wh} is not needed in the ordinary double point case, see Section~\ref{S3} below for a simplified argument in the ordinary double point case.
\sk
Set $\Si\defs{\rm Sing}\,X\sst\PP^{n-1}$. Let $(\dd f)$ and $I_{\Si}\sst\C[x]$ be respectively the Jacobian ideal of $f$ and the graded ideal consisting of homogeneous polynomials vanishing on $\Si$ with $x\eq(x_1,\dots,x_n)$ the natural coordinate system of $\C^n$. Admitting the equality~\eqref{1}, we can verify by using the symmetries in \eqref{3.1} below that
\begin{equation} \label{2}
p_f=p'_f:=\min\bl\{k\ins\N\mid I_{\Si,k}\nes(\dd f)_k\br\}\q\h{if}\q p'_f\less d,
\end{equation}
where $p_f\eq d$ otherwise. Let $q_s$ be the minimal positive integer such that
\begin{equation*}
N_{n,q_s}=\tbinom{q_s+n-1}{n-1}>s:=|\Si|.
\end{equation*}
Note that $N_{n,d}\defs\tbinom{d+n-1}{n-1}$ is the dimension of homogeneous polynomials in $n$ variables of degree $d$. From this {\it strict\one} inequality we then get that $I_{\Si,q_s}\nes0$. Since we take the minimum in \eqref{2}, this non-vanishing implies the {\it upper bound}
\begin{equation} \label{3}
p_f\les q_s\q\h{in the case}\q q_s\less d\mi2\,\,\,\,\h{(so that}\,\,\,(\dd f)_{q_s}\eq0).
\end{equation}

We are interested in whether this bound is {\it sharp\one} or not. The assumption $q_s\less d\mi2$ seems rather reasonable in view of the inequality $N_{n,d}\sgt sn$, which is a (rather trivial) sufficient condition for the existence of nonzero homogeneous polynomials having singularities at given $s$ points. For this problem we have the following.

\begin{iprop} \label{P1}
If the singularities of $X$ are in general position, the equality holds in \eqref{3}.
\end{iprop}

This is shown by using Remark~\ref{R4.2} below, which implies that $I_{\Si,k}\eq0$ for any $k\slt q_s$, see Definition~\ref{D4.1} for {\it general position.} In this paper we prove the following.

\begin{ithm} \label{T2}
Assume $N_{n,m}=\tbinom{m+n-1}{n-1}\ges s$ with $m:=\lfloor d/2\rfloor\mi1$ or $N_{n,d-2}=\tbinom{d+n-3}{n-1}\ges sn$. We assume further the strict inequality in the second case for $(n,d{-}2,s)\eq(3,4,5)$ or $(5,3,7)$ or $(5,4,14)$ in order to avoid the exceptional cases in \cite{AH}. Then for any $s$ points of $\PP^{n-1}$ in sufficiently general position $($see Remark~{\rm\ref{R4.5}} below$\one)$, there is a homogeneous polynomial of degree $d$ such that the associated projective hypersurface has ordinary double points at the given $s$ points and is nonsingular outside them, in particular, the estimate \eqref{3} is optimal in this case.
\end{ithm}

The proof of this assertion is not quite difficult applying a theorem of Alexander and Hirschowitz \cite[Theorem 2]{AH} (see also \cite{BO}) for the second case, if we can formulate a reasonable notion of points {\it in general position.} Here it is not necessary to exclude the case $(n,d{-}2,s)\eq(4,4,9)$ as in \cite{AH}, since $N_{n,d-2}\mi sn$ becomes $-1$ and not 0 as in the other exceptional cases. Note that $n$ in their paper is $n{-}1$ in our paper. According to their theorem, for given $s$ points $p_1,\dots,p_s$ of $\PP^{n-1}$ in ``general position", the vector space $\Sc_{n,d}(p_1,\dots,p_s)$ of homogeneous polynomials in $n$ variables of degree $d$ (together with 0) such that the singular points of the associated hypersurfaces {\it contain\one} these $s$ points has, except the few cases explained above, the expected dimension, which is equal to $\max(N_{n,d}\mi sn,0)$ with $N_{n,d}\eq\binom{d+n-1}{n-1}$, the dimension of homogeneous polynomials in $n$ variables of degree $d$. (Here the definition of ``general position" does not seem to be explicitly stated in these papers.) It is conjectured that Proposition~\ref{P1} and Theorem~\ref{T2} under the first hypothesis can be extended to the case where $X$ has only $A_2$-singularities instead of ordinary double points. One can verify this for each explicit example using a computer if it is not too complicated, see Section~\ref{S5} below. Note that \cite{AH} is not needed for the proof of Theorem~\ref{T2} under the first hypothesis.
\sk
Comparing the first and second hypotheses of Theorem~\ref{T2}, the first seems strictly better than the second only in the case $d\eq4$ with any $n\gess3$ or $d\eq6$ with $n\eq3,4,5$ or $d\eq8$ with $n\eq3\one$; for instance $s'_{3,6}\eq6$, $s''_{3,6}\eq5$, where $s'_{n,d}$, $s''_{n,d}$ denote the maximal integer $s$ such that the first and second inequalities of Theorem~\ref{T2} are satisfied respectively.
\sk
The theorem in \cite{AH} does not seem to imply that there is really a homogeneous polynomial of degree $d$ such that the associated hypersurface has {\it ordinary double points\one} at given $s$ points in ``general position" and is {\it smooth\one} outside them, assuming {\it only\one} $N_{n,d}\sgt sn$, where the opposite inequality $N_{n,d}\less sn$ holds in the exceptional cases with $d\gess3$. One may ask the following.

\begin{iques} \label{Q1}
Is there a good sufficient condition on $s$ (for instance, of the form $N_{n,d}\mi sn\gess\ga$ with $\ga\ins\Z_{\ges4}$) for the existence of a polynomial as above?
\end{iques}

\begin{irem} \label{R1}
Let $s_{n,d}$ be the {\it maximal\one} integer such that $r_{n,d}\defs N_{n,d}\mi s_{n,d}\one n\sgt0$. If $n$ is {\it prime,} we have $r_{n,d}\eq1$ for $d$ divisible by $n$. A non-zero polynomial $f\ins\Sc_{n,d}(p_1,\dots,p_{s_{n,d}})$ is then {\it unique\one} up to non-zero constant multiple, assuming the $s_{n,d}$ points $p_1,\dots,p_{s_{n,d}}$ are in {\it $d$-{\rm AH}-general position,} that is, $\dim\Sc_{n,d}(p_1,\dots,p_{s_{n,d}})\eq r_{n,d}$, see Definition~\ref{D4.1} below. We are interested in whether the associated projective hypersurface $X\sst\PP^{n-1}$ has ordinary double points at the given $s_{n,d}$ points (by calculating the Hessian $\det(\dd_{x_i}\dd_{x_j}f|_{x_n=1})$ at each point) and there are no other singular points (by computing the intersection of the hypersurfaces $\{\dd_{x_i}f\eq0\}$ for $i\ins[1,n]$). We see moreover that for any integer $n$ between 3 and 7, there are integers $d$ between 3 and 18 such that $r_{n,d}\eq1$. Here it is actually enough to calculate the Hilbert polynomial of the Jacobian ring, which gives the sum of the Tjurina numbers of all singular points of $X$, in order to verify the above assertion, see \cite{kosz} and Remark~\ref{R4.9} below.
\end{irem}

\begin{irem} \label{R2}
In the case $(n,d)\eq(3,6)$, the situation is quite exceptional, since we have $N_{3,6}\eq 28$, $s_{3,6}\eq9$, and $r_{3,6}\eq N_{3,3}\mi s_{3,6}\eq1$ (because $N_{3,3}\eq 10$). By these equalities we see that the (essentially unique) nonzero homogeneous polynomial in $\Sc_{3,6}(p_1,\dots,p_9)$ with $p_1,\dots,p_9$ in 6-AH general position is a {\it square\one} of a homogeneous polynomial of degree 3 vanishing at the given 9 points, so it is non-reduced and has non-isolated singularities.
\end{irem}

\begin{irem} \label{R3}
If $(n,d)\eq(4,4)$, we have $N_{4,4}\eq35$, $s_{4,4}\eq8$, and hence $r_{4,4}\eq3$ instead of 1. Any nonzero $f\ins\Sc_{4,4}(p_1,\dots,p_8)$ has {\it non-isolated\one} singularities after setting $x_4\eq1$ with coordinate $x_4$ sufficiently general, assuming the points $p_1,\dots,p_8$ are in 4-AH-general position. Indeed, there are two polynomials $g_1$, $g_2$ of degree 2 which are linearly independent and vanish at the eight points, since $N_{4,2}\mi8\eq2$. Any linear combination of $g_1^2$, $g_1g_2$, $g_2^2$ can be written as $(ag_1\mi bg_2)(a'g_1\mi b'g_2)$ for some $a,b,a',b'\ins\C$, and there is no nontrivial zero divisor in $\C[x]$. So the three polynomials are linearly independent and the vector space $\Sc_{4,4}(p_1,\dots,p_8)$ is spanned by them (because $r_{4,4}\eq3$). Since any linear combination of them is reducible by the above argument, it has non-isolated singularities after setting $x_4\eq1$, since $\dim X\eq2$. (There is a similar situation for $(n,d,s)\eq(3,4,4)$, but $\dim X\eq1$.) This implies that the constant $\ga$ in Question~\ref{Q1} must be at least 4.
\sk
In the case $(n,d)\eq(6,3)$, we have $N_{6,3}\eq56$, $s_{6,3}\eq9$, $r_{6,3}\eq2$. Any $f\ins\Sc_{6,3}(p_1,\dots,p_9)$ seems to have {\it non-isolated\one} singularities after setting $x_6\eq1$, when $p_1,\dots,p_9$ are in 3-AH-general position, see Remark~\ref{R4.9} below. (The case $(n,d)\eq(5,5)$ with $s_{5,5}\eq25$ seems too complicated.)
\end{irem}

\begin{irem} \label{R4}
If $s\eq s_{n,d}\mi1$ with $(n,d)\eq(3,6),(4,4),(6,3)$, any general $f\ins\Sc_{n,d}(p_1,\dots,p_s)$ seems to have ordinary double points at the given $s$ points and is nonsingular outside them after setting $x_n\eq1$ with coordinate $x_n$ sufficiently general by calculating the Hilbert polynomials of the Jacobian rings as in \cite[Remark 3.7]{nwh}. In the case $(n,d)\eq(3,7)$, this seems to hold even when $s\eq s_{3,7}\eq11$ with $r_{3,7}\eq 36\mi 11\one{\cdot}\one 3\eq3$. One might expect that this would be true at least in the case $r_{n,d}\gess4$ (which is closely related to Question~\ref{Q1} with $\ga\eq4$), although there does not seem to be any philosophical reason.
\end{irem}

\tableofcontents
\numberwithin{equation}{section}

\section{Pole order filtrations} \label{S1}
\nin
For a reduced homogeneous polynomial $f$ in $n$ variables of degree $d$, set $X\defs\{f\eq0\}\sst\PP^{n-1}$. Let $b_f(s)$ be the Bernstein-Sato polynomial of $f$. We denote by $\R_f\subset\Q_{>0}$ the set of roots of $b_f(s)$ up to sign. For $p\in\Si\defs{\rm Sing}\,X$, the local Bernstein-Sato polynomial $b_{h_p,p}(s)$ is {\it independent} of a choice of a local defining holomorphic function $h_p$ of $(X,p)$ (see for instance \cite[Remark 4.2(i)]{wh}), and is denoted by $b_{X,p}(s)$. Let $\R_{X,p}\sst\Q_{>0}$ be the set of roots of $b_{X,p}(s)$ up to a sign. Set
\begin{equation*}
\R_X=\mcup_{p\in\Si}\,\R_{X,p}\,\subset\,\R_f,
\end{equation*}
where the last inclusion holds using the above independence. Put
\begin{equation} \label{1.1}
\R_f^0:=\R_f\setminus\R_X\subset\tfrac{1}{d}\one\Z\q\q\h{so that}\q\q\R_f=\R_f^0\sqcup\R_X,
\end{equation}
see for instance \cite[Remark 4.2(ii)]{wh} for the above inclusion. We call $\R_f^0$ the {\it roots of $b_f(s)$ up to a sign supported at the origin.}
\sk
Set $\Ff:=f^{-1}(1)\subset\C^n$ (the Milnor fiber). We define the {\it monodromy eigenspaces\one} of the Milnor cohomology
\begin{equation*}
H^j(\Ff,\C)_{\la}:={\rm Ker}(T_s\mi\la)\subset H^j(\Ff,\C)\q\h{for}\,\,\,\la\in\C^*,
\end{equation*}
where $T_s$ is the semisimple part of the Jordan decomposition of the monodromy. Note that $H^j(\Ff,\C)_{\la}=0$ if $\la^d\nes1$ (since $f$ is homogeneous with degree $d$). These spaces have the {\it pole order filtration} $P$, and we have the following.

\begin{thm}[{\cite[Theorem 2]{bcm}}] \label{T1.1}
For $\al\notin\R_X$, we have
\begin{equation} \label{1.2}
\al\in\R_f^0\q\h{if}\q\Gr_P^p\,H^{n-1}(\Ff,\C)_{\ee(-\al)}\ne 0\q\bl(p=[n-\al]\br),
\end{equation}
where $\ee(-\al):=e^{-2\pi i\al}$, and the converse holds in the case
\begin{equation} \label{1.3}
\al\notin\R_X+\Z_{<0}.
\end{equation}
\end{thm}

\section{Pole order spectral sequences} \label{S2}
\nin
Let $\Om^{\ssb}$ be the graded complex of algebraic differential forms on $\C^n$, whose components are finite free graded modules over $R:=\C[x]$ with $x\eq(x_1,\dots,x_n)$ the coordinate system of $\C^n$. Here the $x_i$ have degree 1 as well as the $\ddd x_i$. For $k\ins\Z$, we have the microlocal {\it pole order spectral sequence}
\begin{equation} \label{2.1}
E_1^{p,q}(f)_k=H^{p+q}_{\df\wedge}(\Om^{\ssb})_{qd+k}\Longrightarrow \Ht^{p+q-1}(\Ff,\C)_{\ee(-k/d)},
\end{equation}
where $H^{\ssb}_{\df\wedge}(\Om^{\ssb})$ denotes the cohomology of the Koszul complex $(\Om^{\ssb},\dfw)$. Its abutment filtration coincides with the pole order filtration $P$ in Theorem~\ref{T1.1} up to the shift of filtration by $\bl[n\mi\kod\br]$, see \cite{kosz}. (Here the reader may assume $k\ins[1,d]$ if he prefers.)
\sk
Assume $\dim\Si\eq0$. Then the $E_1$ and $E_2$-terms of the spectral sequence are given by 
\begin{equation} \label{2.2}
\begin{aligned}
M:=H^n_{\df\wedge}(\Om^{\ssb}),&\q N:=H^{n-1}_{\df\wedge}(\Om^{\ssb})(-d),\\
M^{(2)}:=H^n_{\ddd}(H^{\ssb}_{\df\wedge}(\Om^{\ssb})),&\q N^{(2)}:=H^{n-1}_{\ddd}(H^{\ssb}_{\df\wedge}(\Om^{\ssb}))(-d),\\
\end{aligned}
\end{equation}
since $H^j_{\df\wedge}(\Om^{\ssb})\eq0$ unless $j\ins\{n{-}1,n\}$, see \cite{kosz}. Here $H^{\ssb}_{\ddd}$ denotes the cohomology of a complex whose differential is induced by $\ddd$, and $M,N$ are also denoted by $M^{(1)},N^{(1)}$ respectively. Recall that $(p)$ denotes the shift of grading by $p\in\Z$, that is, $G(p)_k=G_{k+p}$ ($k\ins\Z$) for any graded module $G$.
\sk
We have the following

\begin{thm}[{\cite[Theorem 2]{wh}}] \label{T2.1}
Assume the singularities of $X$ are isolated and weighted homogeneous. Then the pole order spectral sequence degenerates at $E_2$.
\end{thm}

\section{Simplified proof of Theorem~\ref{T1}} \label{S3}
\nin
Set
\begin{equation*}
\begin{aligned}
&\q\q\q\q\q\q\q M'\defs H^0_{\mm}M,\q M''\defs M/M',\\
&\mu_k\defs\dim M,\q\nu_k\defs\dim N,\q\mu'_k\defs\dim M',\q\mu''_k\defs\dim M'',\\
&\de_k\defs\mu_k\mi\nu_{k+d},\q\de'_k\defs\mu'_k,\q\de''_k\defs\mu''_k\mi\nu_{k+d}\eq\de_k\mi\de'_k,
\end{aligned}
\end{equation*}
where $\mm\sst\C[x]$ is the maximal ideal at 0. We have the following.

\begin{thm}[{\cite[Corollaries 1 and 2]{kosz}}] \label{T3.1}
There are symmetries
\begin{equation} \label{3.1}
\de'_k=\de'_{nd-k},\q\de''_k=\de''_{(n-1)d-k}\,\,\,\,(\forall\,k\ins\Z),
\end{equation}
\end{thm}

Combining Theorems~\ref{T1.1} and \ref{T2.1} and using \cite[Theorem 5.3]{kosz}, which implies the injectivity of the $E_1$-differential $\ddd_1\col N_{k+d}\tos M_k$ in the case $\kod$ is not a spectral number (or a root of the reduced local Bernstein-Sato polynomial $b_{X,p}(s)/(s{+}1)$ up to sign, using the weighted homogeneity) of any singular point of $X$, we can deduce the following.

\begin{thm} \label{T3.2}
Assume $X$ has only isolated weighted homogeneous singularities. Then for $k\ins\Z$ with $\kod\,{\notin}\,\R_X$, we have $\kod\ins\R_f^0$ if $\de_k\sgt 0$, and the converse holds in the case $\kod\notin\R_X\pl\Z_{<0}$.
\end{thm}

For the proof of Theorem~\ref{T1} we need also the following (which is due to \cite{DiSt}, \cite{Di} in the ordinary double point case).

\begin{thm}[{\cite[Theorem 9]{odp}}] \label{T3.3}
Assume $X$ has only isolated weighted homogeneous singularities. Then $\nu_{k+d}\eq0$ if $\kod\slt\alt_X\defs\min\widetilde{\R}_X$, where $\widetilde{\R}_X\defs\mcup_{p\in\Si}\,\widetilde{\R}_{X,p}$ with $\widetilde{\R}_{X,p}$ the roots of the reduced local Bernstein-Sato polynomial $b_{X,p}(s)/(s{+}1)$ up to sign, which coincide with the spectral numbers at $p\in\Si$ forgetting multiplicities.
\end{thm}

\begin{proof}[Simplified proof of Theorem~{\rm\ref{T1}}]
We first show that $\de''_k\sgt0$ for any $k\ins\Z\cap\bl[n,\tfrac{(n-1)d}{2}\br)$. This follows from Theorem~\ref{T3.3}, since the $\mu''_k$ are weakly increasing (see \cite{kosz}) and $\mu''_n\eq1$. (Note that $\mu'_n$ cannot be equal to 1.) The assertion is then reduced to Lemma~\ref{L3.4} below using Theorem~\ref{T3.2} together with the second symmetry in Theorem~\ref{T3.1}.
\end{proof}

\begin{lem} \label{L3.4}
With the notation of Introduction, assume $I_{\Si,k_0}\nes(\dd f)_{k_0}$ for some $k_0\ins[1,d{-}1]$. Then $I_{\Si,k}\nes(\dd f)_k$ for any $k\ins[k_0{+}1,d{-}1]$. 
\end{lem}

\begin{proof}
It is enough to consider the case $k\eq d{-}1$ with $k_0\eq d{-}2$, since $(\dd f)_k\eq0$ for $k\less d{-}2$. If the assertion does not hold, there is $g\ins\C[x]_{d-2}$ such that $(\dd f)_{k+d-2}\eq g\one\C[x]_k$ for any $k\ins\Z_{>0}$ (since $\dim_{\C}(\dd f)_{d-1}\eq n$ and the ideal $(\dd f)$ is generated by $(\dd f)_{d-1}$), but this is a contradiction considering the associated projective varieties in $\PP^{n-1}$, see also the proof of \cite[Theorem 5]{wh}.
\end{proof}

\begin{rem} \label{R3.5}
Theorem~\ref{T1} is extended to the case where $X$ has only {\it weighted homogeneous isolated\one} singularities, for instance if $\tau_X\slt\ga_{k_0}$ with $k_0\defs\lceil d\one\alt_X\rceil$, see \cite[Theorem 5]{wh} for a better assertion. Here $\tau_X$ is the sum of the Tjurina or Milnor numbers of the singular points of $X$ and the $\ga_k$ are defined by $\msum_{k\in\Z}\,\ga_k\one t^k\eq(t^d\mi t)^n/(t\mi1)^n$. (Indeed, the inequality implies that $\mu'_k\sgt0$ for any $k\ins [k_0,n\one d\mi k_0]$ using the equality $\mu_k\eq\ga_k\pl\nu_k$.) See also the latest version of \cite{wh} using \cite{DiPo} for the case $n\eq3$.
\end{rem}

\section{Points in general position} \label{S4}

\begin{defi} \label{D4.1}
Let $\iota_m:\PP^{n-1}\into\PP^{N_{n,m}-1}$ be the Veronese embedding of degree $m$, where $N_{n,m}\defs\tbinom{m+n-1}{n-1}$ for $m\ins\Z_{>0}$. We say that $s$ points $p_1,\dots,p_s\ins\PP^{n-1}$ are in $m$-{\it general position\one} if their images in $\PP^{N_{n,m}-1}$ span a linear subspace with dimension equal to $\min(s\mi1,N_{n,m}\mi1)$. We say that the points are in {\it general position\one} if they are in $m$-general position for any $m\ins\Z_{>0}$.
\sk
We say that $s$ points $p_1,\dots,p_s\in\PP^{n-1}$ are in {\it $m$-{\rm AH}-general position\one} if
\begin{equation*}
\dim\Sc_{n,m}(p_1,\dots,p_s)=\max(N_{n,m}\mi sn,0).
\end{equation*}
Here $\Sc_{n,m}(p_1,\dots,p_s)$ is the vector space of homogeneous polynomials $f$ of degree $m$ (together with the polynomial 0) such that their partial derivatives $\dd_{x_i}f$ vanish at the given $s$ points (using the relation $f\eq\tfrac{1}{m}\msum_{i=1}^n\,x_i\dd_{x_i}f$).
\end{defi}

\begin{rem} \label{R4.2}
In the case $m\slt q_s\defs\min\{q\ins\N\,\one|\one\,N_{n,q}\sgt s\}$ (that is, $N_{n,m}\less s$), the points $p_1,\dots,p_s\ins\PP^{n-1}$ are in $m$-general position if and only if there are no nonzero homogeneous polynomials of degree $m$ vanishing on these points (or equivalently the images of the $s$ points are not contained in a hyperplane of $\PP^{N_{n,m}-1}$).
\end{rem}

\begin{rem} \label{R4.3}
Let $\Xi_s$ be the subset of the $s$-fold self-product of $\PP^{n-1}$ consisting of mutually distinct $s$ points of $\PP^{n-1}$ (that is, the complement of the union of various partial diagonals).
\sk
The subset $\Xi_{s,m}$ consisting of $s$ points in $m$-general position is a non-empty Zariski-open subset of $\Xi_s$ for each $m\ins\Z_{>0}$, and a similar assertion holds with $m$-general replaced by general using Proposition~\ref{P4.4} below. The proof of non-emptiness for each $m$ can be done by induction on $s$, since the image of $\iota_m$ is not contained in any hyperplane of $\PP^{N_{n,m}-1}$. (As for the case $n\eq 2$, mutually distinct points are always in general position using Vandermonde's determinant.)
\sk
The subset $\Xi^{\rm AH}_{s,m}$ consisting of $s$ points in $m$-{\rm AH}-general position is also a Zariski-open subset of $\Xi_s$ for any $m\ins\Z_{>0}$. This subset is non-empty except the cases $(n,m,s)\eq(3,4,5)$, $(4,4,9)$, $(5,3,7)$, $(5,4,14)$, according to \cite{AH} (see also \cite{BO}).
\end{rem}

\begin{prop} \label{P4.4}
Mutually distinct $s$ points $p_1,\dots,p_s$ on $\PP^{n-1}$ are in general position if and only if they are in $m$-general position for $m\eq q_s,q_s{-}1$.
\end{prop}

\begin{proof}
It is enough to prove that the points $p_1,\dots,p_s$ are in $m$-general position if they are in $(m{-}1)$-general position with $m{-}1\gess q_s$ or if they are in $(m{+}1)$-general position with $m{+}1\less q_s{-}1$. The second case is shown by using Remark~\ref{R4.2}. For the first case we take projective coordinates $x_1,\dots,x_n$ of $\PP^{n-1}$ such that the $s$ points are contained in the affine space $\C^{n-1}\eq\{x_n\nes0\}\sst\PP^{n-1}$. Then the restriction of the Veronese embedding $\iota_m$ to this affine space can be described by using homogeneous polynomials of degrees at most $m$. So the assertion follows.
\end{proof}

\begin{rem} \label{R4.5}
We say that $(p_1,\dots,p_s)\ins\Xi_s$ is {\it in sufficiently general position\one} if it belongs to a sufficiently small non-empty Zariski-open subset of $\Xi_s$. Here we may assume the subset is invariant under the action of the symmetric group ${\mathfrak S}_s$ by taking the intersection of its images by the action of elements of ${\mathfrak S}_s$ (since the intersection of a finite number of non-empty Zariski-open subsets of an irreducible variety is non-empty).
\sk
In the first case of Theorem~\ref{T2}, a sufficient condition is given by the last one in Definition~\ref{D4.1}. It is actually enough to assume that the $s$ points $p_1,\dots,p_s$ are in $m$-general position for $m\eq1$ and $\lfloor d/2\rfloor\mi1$, where $N_{n,\lfloor d/2\rfloor\mi1}\gess s$ by the assumption itself. Note that one has to consider various intersections of the $V_j$ with the notation in the proof of Theorem~\ref{T2}.
\sk
In the second case a sufficient condition is that the points are in 1-general position and the $s$ points are in $(d{-}2)$-AH-general position. With the notation of Remark~\ref{R4.7} below the last condition is equivalent to that the images of the $s$ points $p_1,\dots,p_s$ by $\iota_{d-2}$, $\iota_{d-2}^{(1)},\dots,\iota_{d-2}^{(n-1)}$ span a linear subspace of dimension $sn{-}1$ in $\PP^{N_{n,d-2}-1}$ in the case $N_{n,d-2}\gess sn$.
\end{rem}

\begin{rem} \label{R4.6}
Let $V'_J$ be the vector space spanned by the $v'_j$ for $j\ins J$, where $J\sst\{1,\dots,r\}$ and $v'_1,\dots, v'_r$ are points in a vector space $V'$. It does not necessarily hold that $V'_I\cap V'_J\eq V'_{I\cap J}$ for $I,J\sst\{1,\dots,r\}$ unless the $v'_1,\dots, v'_r$ are linearly independent. This is the reason for which we need \cite[Theorem 2]{AH}.
\end{rem}

\begin{proof}[Proof of Theorem~{\rm\ref{T2}}]
We first consider the first case. We may assume $m\gess1$ (hence $d\gess4$), since the case $s\eq1$ is easy and well known. We may further assume $s\sgt n$, since the case $s\less n$ with $d\gess 4$ is also easy. Indeed, if $s\eq1$ or $s\less n$ with $d\gess 4$, we can take for instance $f\eq\msum_{i=1}^s\,a_i\bl(\msum_{j\ne i}\,x_i^{d-2}x_j^2\br)\pl \msum_{i=s+1}^n\,a_i\one x_i^d$ for $a_i\ins\C^*$ sufficiently general, considering a family as in the argument below.
\sk
Let $p_1,\dots,p_s$ be $s$ points of $\PP^{n-1}$ in general position. For $j\ins[1,s]$, let $g_j$ be homogeneous polynomials of degree $m$ such that the associated projective hypersurface $\{g_j\eq0\}\sst\PP^{n-1}$ contains the $p_i$ for $i\ne j$. (Note that $N_{n,m}\sgt s{-}1$.) Let $h_j$ be homogeneous polynomials of degree $d\mi2m$ (which is equal to 2 or 3 depending on the parity of $d\one$) such that the associated projective hypersurface $\{h_j\eq0\}\sst\PP^{n-1}$ contains $p_j$ and has an ordinary double point there. Set $f\defs\sum_{j=1}^s\,g_j^2h_j$. The projective hypersurface $X\defs\{f\eq0\}\sst\PP^{n-1}$ has {\it ordinary double points\one} at the $p_j$, and is {\it smooth\one} outside them if the $g_j,h_j$ are sufficiently general (where the $g_j$ are unique up to non-zero constant multiple in the case $s\eq N_{n,m}$). 
\sk
For the proof of this assertion, we consider the Veronese embedding $\iota_m:\PP^{n-1}\into\PP^{N_{n,m}-1}$ and the linear subspace $V_j\sst\PP^{N_{n,m}-1}$ spanned by the $p_i$ for $i\ins[1,s]\stm\{j\}$ in $\PP^{N_{n,m}-1}$ for each $j\ins[1,s]$, which give the fixed part of the linear system containing $g_j$ by restricting to $\iota_m(\PP^{n-1})\sst\PP^{N_{n,m}-1}$. The images of $p_j$ for $j\ins[1,s]$ span a linear subspace of dimension $s{-}1$ in $\PP^{N_{n,m}-1}$. Studying intersections of the $V_j$, we then see that 
\begin{equation*}
\mcap_{j\in[1,s]}\{g_jh_j\eq0\}=\mcup_{j\in[1,s]}\{p_j\}\q\h{in}\,\,\,\,\PP^{n-1},
\end{equation*}
if the $g_j,h_j$ are sufficiently general. Here the $h_j$ are chosen after the $g_j$.
\sk
We now consider the family
\begin{equation*}
\X:=\bl\{G\defs\msum_{j=1}^s\,t_j\one g_j^2h_j=0\br\}\subset\PP^{n-1}{\times}S,
\end{equation*}
where $S\defs\C^s$ with coordinates $t_1,\dots,t_s$. Let $\widetilde{\PP}^{n-1}$ denote the blowup of $\PP^{n-1}$ along the union of $\{p_j\}$ for $j\ins[1,s]$. Let $\widetilde{\X}$ be the proper transform of $\X$ in $\widetilde{\PP}^{n-1}{\times}S$. We see that $\widetilde{\X}$ is smooth (since $\dd_{t_j}G\eq g_j^2h_j$), and so is a general fiber of the projection $\widetilde{\X}\tos S$. The assertion is thus proved in the first case.
\sk
The argument is similar for the second case in view of Remark~\ref{R4.7} below where $d$ must be replaced by $d{-}2$. We may assume $d{-}2\gess2$, since we get $s\eq1$ from the inequality if $d\eq3$. We may then assume $s\sgt n$, take $g_j\in\Sc_{n,d-2}(p_1,\dots,\widehat{p_j},\dots,p_s)$ for $j\ins[1,s]$ (since $N_{n,d-2}\gess sn\sgt(s{-}1)n$), and argue as above. This finishes the proof of Theorem~\ref{T2}.
\end{proof}

\begin{rem} \label{R4.7}
Assume the $s$ points are contained in $\{x_n\ne0\}\sst\PP^{n-1}$ and $d\gess2$. Fixing the projective coordinates $x_1,\dots,x_n$, we have for $i\ins[1,n{-}1]$ the embeddings
\begin{equation*}
\iota_d^{(i)}:\PP^{n-1}\ni[x_1\col\dots\col x_n]\mapsto[\dd_{x_i}x^{\nu}x_n]_{|\nu|=d}\in\PP^{N_{n,d}-1},
\end{equation*}
where $|\nu|\defs\msum_{i=1}^n\,\nu_i$. We use the projective coordinates $(y_{\nu})_{\nu|=d}$ of $\PP^{N_{n,d}-1}$ associated with the projective coordinates $x_1,\dots,x_n$ of $\PP^{n-1}$ so that the Veronese embedding $\iota_d$ can be given by $y_{\nu}\eq x^{\nu}$ for any $\nu\ins\N^n_{[d]}$, where $\N^n_{[d]}\defs\{\nu\ins\N^n\,{|}\,|\nu|\eq d\}$. The condition that the partial derivatives of a homogeneous polynomial vanish at the given $s$ points is given by the orthogonal relations to the images of the $s$ points by $\iota_d,\iota_d^{(1)},\dots,\iota_d^{(n-1)}$, using the equality $f\eq\msum_{i=1}^n\,x_i\one\dd_{x_i}f$ and restricting to $x_n\eq1$. Indeed, if $f\eq\msum_{\nu\in\N^n_{[d]}}a_{\nu}x^{\nu}$, we have
\begin{equation*}
\dd_{x_i}f\eq\msum_{\nu\in\N^n_{[d]}}a_{\nu}\dd_{x_i}x^{\nu}\q\h{for}\,\,\,i\ins[1,n{-}1],
\end{equation*}
even after restricting to $x_n\eq1$. The $a_{\nu}$ for $\nu\ins\N^n_{[d]}$ can be viewed as projective coordinates of the dual projective space $\check{\PP}^{N_{n,d}-1}$
\sk
For a subset $I\sst\{1,\dots,s\}$, let $V_I\sst\PP^{N_{n,d}-1}$ be the linear subspace spanned by $\iota_d(p_j)$, $\iota_d^{(1)}(p_j),\dots,\iota_d^{(n-1)}(p_j)$ for $j\ins I$. We have
\begin{equation*}
V_I\cap V_J\eq V_{I\cap J}\q\h{for}\,\,\,\,I,J\sst\{1,\dots,s\},
\end{equation*}
using \cite[Theorem 2]{AH}, see Remark~\ref{R4.6}.
\sk
Assume $p_1\eq(0,\dots,0)$, and $p_{j+1}\ins\mcap_{i\ne j}\{x_i\eq0\}\sst\PP^{n-1}\stm\{x_n\eq0\}$ for $j\ins[1,n{-}1]$. Then
\begin{equation*}
x_i^2x_n^{d-2}\in\Sc_{n,d}(p_1,\dots,p_j)\q\h{for any}\,\,i\in[j,n{-}1],
\end{equation*}
hence we get
\begin{equation*}
\dim\bl(V_{\{1,\dots,j\}}\cap\iota_d(\PP^{n-1})\stm\{x_n\eq0\}\br)<j\q\h{for any}\,\,j\ins[1,n{-1}],
\end{equation*}
where the coordinate $x_n$ can be changed as long as $p_j\notin\{x_n\eq0\}$ for any $j\ins[1,s]$.
\sk
Note that, for any subset $I\sst\{1,\dots,s\}$ with $|I|\sgt s\mi n\sgt0$, there is a subset $I'\sst I$ with $|I'|\eq|I|\mi s\pl n$ and such that the $p_j$ for $j\ins I'$ span a linear subspace of dimension $|I'|\mi1$ in $\PP^{n-1}$ by decreasing induction on $|I|$ if the $s$ points are in 1-general position with $s\sgt n$.
\end{rem}

\begin{exam} \label{E4.8}
Assume $0\slt s\less n$ and $s$ points $p_1,\dots,p_s$ span a linear subspace of $\PP^{n-1}$ of dimension $s{-}1$. Then the $s$ points are in general position and we have $p_f\eq q_s\eq2$ if $s\eq n$, and $p_f\eq q_s\eq1$ otherwise. Assume $0\slt s\less n$ and $s$ points span a linear subspace of $\PP^{n-1}$ of dimension strictly smaller than $s{-}1$. Then the $s$ points are not in general position and we have always $p_f\eq1$.
\sk
As a concrete example for the first case with $s\eq n$, we have $f\eq\msum_{1\les i<j\les n}\,x_i^2x_j^2$ for any $n\gess3$. For the second case with $s\eq n$, one may consider for instance
\begin{equation*}
\begin{aligned}
&f=x^2y^2{+}x^2z^2{+}y^2z^2{-}(x{+}y{+}z)xyz{+}
(x^2{+}y^2{+}z^2)w^2{-}w^4\q\q\h{or}\\
&f=x^2y^2{+}x^2z^2{+}y^2z^2{-}(x{+}y{+}z)xyz{+}
(x^2{+}y^2{+}z^2)w^2{+}2(x^2{+}y^2{+}z^2{+}w^2)v^2{+}v^4,
\end{aligned}
\end{equation*}
where $n\eq4$ or 5. These can be verified using Macaulay2 as follows.
\ms
\vbox{\q\q\small\tt\verb#R=QQ[x,y,z,w,v]; f=x^2*y^2+x^2*z^2+y^2*z^2-(x+y+z)*x*y*z+#
\sk\q\q
\verb#(x^2+y^2+z^2)*w^2+2*(x^2+y^2+z^2+w^2)*v^2+v^4;#
\sk\q\q
\verb#I=radical ideal(jacobian ideal f); decompose I#}
\skn
see also the code at the end of Appendix in \cite{wh} calculating the $\mu'_k$ and $\de_k$.
\end{exam}

\begin{rem} \label{R4.9}
One can get homogeneous polynomials whose derivatives vanish at given points by using Singular \cite{Sing} as follows (where $(n,d,s)\eq(4,4,8)$).
\ms
\vbox{\scriptsize\tt\verb#LIB "general.lib"; LIB "control.lib";#
\sk
\verb#ring S=0,(x,y,z,w),ds; int a,b,i,j,k,p,q,d,n,r,s,N; list I; n=4; d=4;#
\sk
\verb#poly f; N=int(factorial(d+n-1)div factorial(n-1)div factorial(d));#
\sk
\verb#intmat L[N][n]; intmat D[n][n]; for(i=1; i<=n; i++) {D[i,i]=1;}#
\sk
\verb#for(i=0; i<=d; i++) {for(j=0; j<=d-i; j++) {for(k=0; k<=d-i-j; k++) {I=d,i,j,k;#
\sk
\verb#r=(I[1]-I[2]+2)*(I[1]-I[2]+1)*(I[1]-I[2]) div 6+ (I[1]-I[2]-I[3]+1)*(I[1]-I[2]-I[3])#
\sk
\verb#div 2+I[1]-I[2]-I[3]-I[4]+1; L[r,1]=i; L[r,2]=j; L[r,3]=k; L[r,4]=1;}}}#
\sk
\verb#s=8; intmat P[s][n-1]=0,0,0, 1,0,0, 0,1,0, 0,0,1, -1,2,-3, 4,-3,2, 2,-3,-6, -2,4,1;#
\sk
\verb#intmat R[s*n][N]; for (i=1;i<=s; i++){ for (j=1; j<=N; j++) {for (k=1; k<=n; k++)#
\sk
\verb#{if (L[j,k]>0){R[n*(i-1)+k,j]=L[j,k]; for (p=1;p<=n-1; p++) {R[n*(i-1)+k,j]=#
\sk
\verb#R[n*(i-1)+k,j]*P[i,p]^(L[j,p]-D[k,p]);}}}}} matrix K=rightKernel(R); r=size(K)div N;#
\sk
\verb#matrix F[r][n]; for(p=1; p<=r; p++) {for(i=1; i<=N; i++){F[p,n]=F[p,n]+K[i,p]*#
\sk
\verb#x^(L[i,1])*y^(L[i,2])*z^(L[i,3])*w^(d-L[i,1]-L[i,2]-L[i,3]);} F[p,1]=diff(F[p,n],x);#
\sk
\verb#F[p,2]=diff(F[p,n],y); F[p,3]=diff(F[p,n],z); for(j=1; j<=s; j++){for(k=1; k<=n; k++)#
\sk
\verb#{if(subst(F[p,k],x,P[j,1],y,P[j,2],z,P[j,3],w,1)!=0){sprintf("Nonzero at %s,%s, %s",#
\sk
\verb#p,j,k);}}}} a=2; b=3; for(i=1; i<=r; i++) {f=f+(-1)^i*(a*i^2-b*i+1)*F[i,n]; F[i,n];}#
\sk
\verb#r; ideal J=jacob(f); size(kbase(std(J),n*d-n)); size(kbase(std(J),n*d-n+1));#}
\msn
This gives three homogeneous polynomials in 4 variables of degree 4 which are linearly independent and whose partial derivatives vanish at the given eight points. We see that these eight points are in 4-AH-general position, since the size of the matrix {\small\tt\verb#K#} divided by $N_{n,d}$ is the correct one, that is, 3 (which is given by the third last output for {\small\tt\verb#r#}). One can verify that any linear combination of these polynomials restricted to $w\eq1$ does not have an ordinary double point at 0 by applying {\small\tt\verb#subst(f,z,z-33y,y,y+12x,w,1);#} to {\small\tt\verb#f=F[1,n],F[2,n],F[3,n]#}, since their degree 2 terms become respectively $z^2,yz,y^2$. (This is quite compatible with Remark~\ref{R3}.) Any linear combination of them has non-isolated singularities by Remark~\ref{R3}, and this agrees with a computation of the Hilbert polynomial of the Jacobian ring in the last part of the above code, see also \cite[Remark 3.7]{nwh}. Indeed, the last two outputs must be the same and coincide with the sum of the Tjurina numbers if $X\defs\{f\eq0\}\sst\PP^{n-1}$ has only isolated singularities. (Note that this does not imply the same property for {\it any\one} other $s$ points in 4-AH-general position as shown at the end of this remark.) When we calculate the case $(n,d,s_{4,d})\eq(4,5,13)$ by adding the input, we get the expected total Tjurina number 13 twice after $r\eq4$.
\sk
If one replaces {\small\tt\verb#s=8#} with {\small\tt\verb#s=7#} (without changing {\small\tt\verb#d=4#}), then one obtains seven polynomials, since the first seven points of the above eight points are also in 4-AH-general position. Here {\small\tt\verb#f#} in {\small\tt\verb#jacob(f)#} must be a ``general" linear combination of them. If it is appropriate, the last two outputs would be 7, 7, which is the number of ordinary double points in this case.
\sk
It may be possible to compute the exceptional cases in \cite[Theorem 2]{AH}, for instance $(n,d,s)\eq(4,4,9)$ (where $n$ in our paper is equal to $n{+}1$ in their paper) if one can add some points to the input {\it appropriately\one} so that the correct value of {\small\tt\verb#r#} in the third last output (which is 1) is obtained.
\sk
If one likes to calculate the case $(n,d,s)\eq(6,3,9)$, the above code must be modified appropriately by replacing some part with (or adding) the following\,:
\ms
\vbox{\scriptsize\tt\verb#ring S=0,(x,y,z,u,v,w),ds; int a,b,i,j,k,p,q,d,n,r,s,N; list I; n=6; d=3;#
\sk
\verb#for(i=0; i<=d; i++) {for(j=0; j<=d-i; j++) {for(k=0; k<=d-i-j; k++)#
\sk
\verb#{for(p=0; p<=d-i-j-k; p++) {for(q=0; q<=d-i-j-k-p; q++) {I=d,i,j,k,p,q;#
\sk
\verb#r=(I[1]-I[2]+4)*(I[1]-I[2]+3)*(I[1]-I[2]+2)*(I[1]-I[2]+1)*(I[1]-I[2]) div 120+ #
\sk
\verb#(I[1]-I[2]-I[3]+3)*(I[1]-I[2]-I[3]+2)*(I[1]-I[2]-I[3]+1)*(I[1]-I[2]-I[3]) div 24+#
\sk
\verb#(I[1]-I[2]-I[3]-I[4]+2)*(I[1]-I[2]-I[3]-I[4]+1)*(I[1]-I[2]-I[3]-I[4]) div 6+#
\sk
\verb#(I[1]-I[2]-I[3]-I[4]-I[5]+1)*(I[1]-I[2]-I[3]-I[4]-I[5]) div 2+I[1]-I[2]-I[3]-I[4]#
\sk
\verb#-I[5]-I[6]+1; L[r,1]=i; L[r,2]=j; L[r,3]=k; L[r,4]=p; L[r,5]=q; L[r,6]=1;}}}}}#
\sk
\verb#s=9; intmat P[s][n-1]=0,0,0,0,0, 1,0,0,0,0, 0,1,0,0,0, 0,0,1,0,0, 0,0,0,1,0,#
\sk
\verb#0,0,0,0,1, -1,2,-3,4,-5, 2,-5,4,-3,-1, -3,4,-1,5,-2;#
\sk
\verb#*u^(L[i,4])*v^(L[i,5])*w^(d-L[i,1]-L[i,2]-L[i,3]-L[i,4]-L[i,5])#
\sk
\verb#F[p,4]=diff(F[p,n],u); F[p,5]=diff(F[p,n],v);   u,P[j,4],v,P[j,5],#}
\msn
In the case $n\eq3$ the changes may be as follows.
\ms
\vbox{\scriptsize\tt\verb#ring S=0,(x,y,w),ds; int a,b,i,j,k,p,q,d,n,r,s,N; list I; n=3; d=10;#
\sk
\verb#for(i=0; i<=d; i++) {for(j=0; j<=d-i; j++) {I=d,i,j; r=(I[1]-I[2]+1)*#
\sk
\verb#(I[1]-I[2]) div 2+ I[1]-I[2]-I[3]+1; L[r,1]=i; L[r,2]=j; L[r,3]=1;}}#
\sk
\verb#s=21; intmat P[s][n-1]=0,0, 4,3, 4,1, 4,-1, 4,-3, 2,4, 2,2, -3,3, 2,-2, 2,-4,#
\sk
\verb#-3,1, -3,-1, -3,-3, -2,4, 1,3, 1,1, 1,-3, 2,0, -2,2, -2,-2, -2,-4;#
\sk
\verb#x^(L[i,1])*y^(L[i,2])*w^(d-L[i,1]-L[i,2])#}
\msn
Here {\scriptsize\tt\verb#F[p,3]=diff(F[p,n],z); z,P[j,3],#} {\it must be deleted.}
One can then calculate also the cases $(d,s_{3,d})\eq(9,18),(8,14),(7,11),(6,9),(5,6),(4,4)$ by changing the input appropriately. Here $r_{3,d}\eq1$ if $\tfrac{d}{3}\ins\Z$, and 3 otherwise. The case $(d,s_{3,d})\eq(7,11)$ seems rather complicated, since we may get an example having isolated singularities whose total Tjurina number is not the wanted one although the value of {\small\tt\verb#r#} is correct. (It is not necessarily easy to find $s_{n,d}$ points in $d$-AH-general position. This can be judged by looking at the value of {\small\tt\verb#r#} which is given in the third last output.)
\sk
It turns out however that the example for the case $(n,d,s)\eq(3,7,11)$ in Remark~\ref{R4} had no problem, hence the property that a general $f\ins\Sc_{n,d}(p_1,\dots,p_{s_{n,d}})$ has ordinary double points at $p_1,\dots,p_{s_{n,d}}$ and is nonsingular outside them after setting $x_n\eq1$ depends on the points $p_1,\dots,p_{s_{n,d}}$ in $d$-AH-general position. We note here some example with $(n,d,s)\eq(3,7,11)$:
\ms
\vbox{\scriptsize\tt\verb#P[s][n-1]=0,0, 1,2, 1,-2, -2,-1, 1,4, 2,2, -3,3, 2,-2, 2,-4, -3,1, -3,-1;#
\sk
\verb#P[s][n-1]=0,0, 1,2, 1,-2, -2,-1, 2,4, 2,2, -3,3, 2,-2, 2,-4, -3,1, -3,-1;#}
\msn
These two give the total Tjurina numbers 11 and 13 respectively although $r\eq3$ for both. (Here the integers {\small\tt\verb#a,b#} for a linear combination of the {\small\tt\verb#F[i,n]#} must be modified sufficiently in order to get a general one.) This means that any $s$ points in $d$-AH-general position are not necessarily in {\it sufficiently\one} general position.
\end{rem}

\section{Extension to singularities of type A{\scriptsize c}} \label{S5}
It seems quite nontrivial to extend Proposition~\ref{P1} and Theorem~\ref{T2} to the case where $X$ has only $A_c$-{\it singularities\one} with Milnor number $c$ for a fixed $c\ins\Z_{>1}$ by using Remark~\ref{R3.5} even for $c\eq2$. Here the singular locus $\Si$ of $X$ is {\it non-reduced,} and we have to indicate the ``direction" of nilpotent elements of $\OO_{\Si,p_j}$ for each $p_j\ins\Si$. This can be done by choosing a local analytic vector field $\xi_j$ not vanishing at $p_j$: A germ of a holomorphic function $g\ins\OO_{\PP^{n-1},p_j}$ belongs to the ideal $\I_{\Si,p_j}$ of $\Si$ if and only if
\begin{equation} \label{5.1}
(\xi_j^k(g))(p_j)\eq0\q\h{for any}\,\,\,\,k\ins[0,c{-}1],
\end{equation}
where $\xi_j^k(g)\ins\OO_{\PP^{n-1},p_j}$.

\begin{rem} \label{R5.1}
In \eqref{5.1} we may replace $\xi_j$ with $u\xi_j$ for an invertible $u\ins\OO^*_{\PP^{n-1},p_j}$, using the relation $[\xi_j,u]\eq\xi_j(u)$. We may also replace $\xi_j$ with $\xi_j\pl\eta$ if $\eta$ vanishes at any point of the integral curve $C_j$ of $\xi_j$ passing through $p_j$. (Indeed, $[\xi_j,\eta]$ vanishes at any point of $C_j$, taking local coordinates $y_1,\dots,y_{n-1}$ around $p_j$ such that $\xi_j\eq\dd_{y_1}$, $C_j\eq\mcap_{2\les i\les n-1}\{y_i\eq0\}$.) So we may replace $\xi_j$ with $\xi'_j$ if their integral curves passing through $p_j$ coincide.
\end{rem}

Set
\begin{equation*}
q_{c,s}\defs\min\{q\ins\N\,|\,N_{n,q}\sgt cs\}.
\end{equation*}
Theorem~\ref{T1} can be extended to the case where $X$ has only $A_c$-singularities assuming $\tau_X\eq cs$ is not very large, see Remark~\ref{R3.5}, hence the equality~\eqref{1} holds with $p_f$ replaced by some positive integer $p_{f,c}$ as follows:
\begin{equation*}
\R_f=\tfrac{1}{d}\one\bl(\Z\cap[n,nd\mi n\mi p_{f,c}]\br)\cup\R_X,
\end{equation*}
As in \eqref{2} and \eqref{3}, we get the {\it upper bound}
\begin{equation} \label{5.2}
p_{f,c}\less q_{c,s}\q\h{in the case}\,\,\,q_{c,s}\less d{-}2.
\end{equation}
It seems however quite unclear whether this is optimal, especially when $c$ is close to $d$. The difficulties seem twofold.
\sk
Set $m_c\defs\lfloor c/2\rfloor$, $m'_c\defs c\pl1\pl(m_c{+}1)(n{-}2)$.
For $h\ins\OO_{\PP^{n-1},p_j}$, the condition that $h$ has at least $A_c$-singularity along $\xi_j$ at $p_j$ can be given by the following $m'_c$ conditions:
\begin{equation} \label{5.3}
(\xi_j^k(h))(p_j)\eq(\xi_j^{k'}\dd_{y_i}(h))(p_j)\eq0\q\h{for}\,\,k\ins[0,c],\,k'\ins[0,m_c],\,i\ins[2,n{-}1],
\end{equation}
calculating the {\it Newton boundary\one} of the $A_c$-singularity, where $y_1,\dots,y_{n-1}$ are analytic local coordinates of $\PP^{n-1}$ at $p_j$ such that $\xi_j\eq\dd_{y_1}$. Indeed, for $h\eq\msum_{\nu\in\N^{n-1}}\,c_{\nu}y^{\nu}$, these conditions mean that $c_{\nu}\eq0$ if $|\nu|\eq\nu_1\less c$ or $|\nu|\eq\nu_1{+}1\less m_c{+}1$.

\begin{rem} \label{R5.2}
In \eqref{5.3} one may replace the $\dd_{y_i}$ with vector fields $\eta_{j,i}$ around $p_j$ for $i\ins[2,n{-}1]$ such that the $\eta_{j,i}$ and $\xi_j$ generate $\Theta_{\PP^{n-1}}$, where the latter denotes the module of holomorphic vector fields around $p_j$ to simplify the notation. Indeed, $\xi_j^a\Theta_{\PP^{n-1}}\xi_j^{q-a}\sst\D_{\PP^{n-1}}$ is independent of $a\ins[0,q]$ modulo $\msum_{k=0}^{q-1}\,\Theta_{\PP^{n-1}}\xi_j^k$, since $\Theta_{\PP^{n-1}}$ is a Lie algebra. The condition~\eqref{5.3} may be replaced with 
\begin{equation*}
h(p_j)\eq(\xi_j^k(h))(p_j)\eq(\eta_{j,i}\xi_j^{k'}(h))(p_j)\eq0,
\end{equation*}
for $k\ins[m_c{+}2,c],\,k'\ins[0,m_c],\,i\ins[1,n{-}1]$, if vector fields $\eta_{j,i}$ for $i\ins[1,n{-}1]$ generate $\Theta_{\PP^{n-1}}$ around $p_j$. Here $\eta_{j,1}$ is not necessarily equal to $\xi_j$, and we assume this for the relation with the above condition.
\end{rem} 

Let $\Sc_{n,d,c}(p_1,\xi_1,\dots,p_s,\xi_s)$ be the vector space of homogeneous polynomials in $n$ variables of degree $d$ (together with 0) such that the associated projective hypersurface $X\sst\PP^{n-1}$ has at least $A_c$-singularities along $\xi_j$ at $p_j$ for any $j\ins[1,s]$. By the above argument we have the inequality
\begin{equation} \label{5.4}
\dim\Sc_{n,d,c}(p_1,\xi_1,\dots,p_s,\xi_s)\ges\max(N_{n,d}-s\one m'_c,0),
\end{equation}
but the answer to the following fundamental question seems quite nontrivial. 

\begin{ques} \label{Q5.3}
Does the hypersurface defined by a general member of $\Sc_{n,d,c}(p_1,\xi_1,\dots,p_s,\xi_s)$ have $A_c$-singularities at $p_1,\dots,p_s$ with no other singular points if $N_{n,d'}\mi s\one m'_c\gess\ga$ for some small positive integer $\ga$, assuming $(p_1,\xi_1),\dots,(p_s,\xi_s)$ are sufficiently general?
\end{ques}

Set $d'\defs d\mi d_{n,c}$, where $d_{n,c}$ is the {\it minimal\one} positive integer such that there is a projective hypersurface of degree $d_{n,c}$ in $\PP^{n-1}$ which has one $A_c$-singularity and is smooth outside it. (Note that the automorphism group of $\PP^{n-1}$ acts transitively on $\PP^{n-1}$.) We have $d_{n,2}\eq3$ for any $n\ges3$, however, it seems rather difficult to determine $d_{n,c}$ in general, especially when $c$ is large. One might have to {\it replace\one} $d_{n,c}$ with the degree of a hypersurface whose singularities satisfy the above condition, forgetting the minimality. Contrary to the ordinary double point case \cite{AH}, the answer to the following question seems quite unclear.

\begin{ques} \label{Q5.4}
Does the equality hold in \eqref{5.4} if $N_{n,d'}\mi s\one m'_c\gess\be$ with $\be$ a small non-negative integer, assuming $(p_1,\xi_1),\dots,(p_s,\xi_s)$ are sufficiently general?
\end{ques}

In the case $n\eq3$, $c\eq2$ (hence $m'_c\eq5$) with $(d,s)\eq(3,2)$ so that $N_{n,d}\eq s\one m'_c$, the left-hand side of \eqref{5.4} seems to be 1, although the right-hand side vanishes. Note however that the {\it equality\one} holds for $(d,s)\eq(4,3)$, $(8,9)$, $(9,11)$ among others. Here it does not seem easy to assure that the input points and vectors are {\it sufficiently general.} As a consequence it is unclear whether the {\it linear independence\one} of the conditions holds as in the ordinary double point case, assuming $N_{n,d}\gess s\one m'_c$, see Remark~\ref{R4.6} for its importance. This problem however does not seem so crucial as the one in Question~\ref{Q5.5} below, since we can {\it restrict\one} to relatively small $s$ (in other words, replace $d$ by a bigger degree with $s$ fixed).
\sk
Let $\Sc'_{n,e,c}(p_1,\xi_1,\dots,p_s,\xi_s)$ be the vector space of homogeneous polynomials in $n$ variables of degree $e$ (together with 0) such that the corresponding section of $\OO_{\PP^{n-1}}(e)$ belongs to the ideal of the non-reduced singular locus of $A_c$-singularities along $\xi_j$ at $p_j$. This is independent of a local trivialization of $\OO_{\PP^{n-1}}(e)$. We have the inequality
\begin{equation} \label{5.5}
\dim\Sc'_{n,e,c}(p_1,\xi_1,\dots,p_s,\xi_s)\ges\max(N_{n,e}\mi cs,0),
\end{equation}
but the answer to the following question seems unclear.

\begin{ques} \label{Q5.5}
Does the equality hold in \eqref{5.5} for any integer $e\slt q_{c,s}$ (so that the right-hand side of \eqref{5.5} is equal to 0), assuming $(p_1,\xi_1),\dots,(p_s,\xi_s)$ are sufficiently general?
\end{ques}

This equality is {\it indispensable\one} to extend the proof of the {\it sharpness\one} of the upper bound to the $A_c$-singularity case, and was {\it not\one} difficult in the ordinary double point case, that is, if $c\eq1$, see Proposition~\ref{P4.4}. This could be a {\it serious\one} problem, since the condition $e\slt q_{c,s}$ is {\it completely determined by\one} $n,c,s$, and {\it cannot be replaced.} No counterexample to Question~\ref{Q5.5} is nevertheless known for the moment as far as computations are made including the cases $(e,c,s)\eq(3,5,2)$, $(3,2,5)$, $(4,5,3)$, $(4,3,5)$ with $n\eq3$ (where $N_{n,e}\mi cs\eq0$) among many others. It seems however quite unclear how this can be proved in general even in the case $c\eq2$ by modifying an argument in \cite{AH}.

\begin{rem} \label{R5.6}
By an argument similar to the proof of Theorem~\ref{T2}, it is possible to extend Proposition~\ref{P1} and Theorem~\ref{T2} for given $n,d,c,s$ if Questions~\ref{Q5.4} and \ref{Q5.5} have positive answers for these numbers. Here $d_{n,c}$ may be replaced with the degree of a hypersurface whose singularities are described as above forgetting the minimality (since it does not necessarily seem easy to determine it in general). This is a generalization of Theorem~\ref{T2} for the {\it second\one} hypothesis.
\end{rem}

Corresponding to the {\it first\one} assumption of Theorem~\ref{T2}, set $d''\defs\lfloor(d\mi d_{n,c})/2\rfloor$. Here $d_{n,c}$ may be replaced by an integer obtained by forgetting the minimality as above. We have the following.

\begin{ques} \label{Q5.7}
Does the equality hold in \eqref{5.5} for $e\eq d''$ if $N_{n,d''}\mi cs\gess0$
(for instance if $d''\gess q_{c,s}$ so that $N_{n,d''}\mi cs\sgt0$), assuming $(p_1,\xi_1),\dots,(p_s,\xi_s)$ are sufficiently general?
\end{ques}

\begin{rem} \label{R5.8}
One can extend Proposition~\ref{P1} and Theorem~\ref{T2} for the {\it first\one} hypothesis if Questions~\ref{Q5.5} and \ref{Q5.7} have positive answers. Here {\it only\one} the inequality \eqref{5.5} is concerned, since we have $f^2\ins\Sc_{n,2e,c}(p_1,\xi_1,\dots,p_s,\xi_s)$ for $f\ins\Sc'_{n,e,c}(p_1,\xi_1,\dots,p_s,\xi_s)$, calculating the Newton boundaries. Comparing the case $c\eq2$ with \cite{AH}, one may have the following.
\end{rem}

\begin{conj} \label{C5.9}
If $c\eq2$, one can extend Proposition~\ref{P1} and Theorem~\ref{T2} under the {\it first\one} hypothesis to the case where $X$ has only $A_2$-singularities (with $d_{n,2}\eq3$).
\end{conj}

There is no counterexample as far as computations are made. Note that \cite{AH} is not needed for the proof of Theorem~\ref{T2} under the {\it first\one} hypothesis. Assume for instance $(n,d,c,s)\eq(3,9,2,5)$, where $d''\eq3$, $q_{2,5}\eq4$, since $N_{3,3}\eq10$. It is then enough to study Question~\ref{Q5.5} for $(n,e,c,s)\eq(3,3,2,5)$, since the case $e\eq2$ or 1 is trivial. Running the code for Singular \cite{Sing} as below, we obtain $r\eq0$ as expected. Here the input points and vectors could be modified $2^9$ times, although the computation stops as soon as it hits the case with $r\eq\max(N_{n,e}\mi cs,0)$ ten times. (Note that the good cases are generic.) Proposition~\ref{P1} and Theorem~\ref{T2} for the first hypothesis are thus generalized when $(n,d,c,s)\eq(3,9,2,5)$. If $(n,d,c,s)\eq(3,11,2,5)$, we have to run the code for $(n,e,c,s)\eq(3,4,2,5)$ and $(3,3,2,5)$. Questions~\ref{Q5.5} and \ref{Q5.7} for $c\eq2$ are examined by running (variants of) the code in the cases where $(e,s)\eq(2,3)$, $(3,5)$, $(5,10)$, $(5,11)$, $(6,14)$ with $n\eq3$, $(e,s)\eq(3,10)$, $(4,17)$, $(4,18)$ with $n\eq4$, and $(e,s)\eq(3,17)$, $(3,18)$, $(4,35)$ with $n\eq5$ (satisfying $|N_{n,e}\mi cs|\less1$) among others. Some of these may be related to the exceptional cases in \cite{AH}. We have the flexibility of {\it directions\one} of vector fields in our case compared with that paper. Note that $\mm^2_{p_j}\sst\I_{\Si,p_j}\sst\mm_{p_j}$ with $\mm_{p_j}\sst\OO_{\PP^{n-2},p_j}$ the maximal ideal if $c\eq2$, and $\mm^2_{p_j}$ cannot be equal to $\I_{\Si,p_j}$ for any singularity of $X$, since it is not a complete intersection, that is, not generated by $n{-}1$ elements. (This code is not necessarily good for the case $c\nes2$.)
\ms
\vbox{\scriptsize\tt\verb#LIB "control.lib";ring T=0,(x,y),ds; int a,b,c,e,i,j,k,l,m,n,p,q,r,s,u,C,M,N;#
\sk
\verb#a=3; b=2; u=4; n=3; e=3; c=2; s=5; intmat P[s][n-1]=0,0, 6,0, 1,9, 13,8, -7,12;#
\sk
\verb#list U,V,W; for (i=1; i<=s; i++) {V[i]=(-1)^i*2*i;} intmat B[e+n][e+n]; B[1,1]=1;#
\sk
\verb#for (i=2; i<=e+n; i++) {for (j=1;j<=i; j++) {if(j>1){B[i,j]=B[i-1,j]+B[i-1,j-1];}#
\sk
\verb#else {B[i,j]=B[i-1,j];}}} N=B[e+n,n]; intmat L[N][n]; list I; for(i=0; i<=e; i++)#
\sk
\verb#{for(j=0; j<=e-i; j++) {I=e,i,j; r=(I[1]-I[2]+1)*(I[1]-I[2]) div 2+I[1]-I[2]-I[3]#
\sk
\verb#+1; L[r,1]=i; L[r,2]=j;}} matrix R[s*c][N]; intmat Q[s][n-1]; matrix K; poly f,g;#
\sk
\verb#M=2^(3*s-6); for(k=0; k<M&&C<10; k++) {m=k; for (i=1; i<=s; i++){U[i]=V[i]+m%2*u;#
\sk
\verb#if (i>3) {Q[i,1]=P[i,1]+m div 2%2*a; Q[i,2]=P[i,2]+m div 4%2*b; m=m div 8;} else{#
\sk
\verb#Q[i,1]=P[i,1]; Q[i,2]=P[i,2]; m=m div 2;} for (j=1; j<=N; j++) {for (q=0; q<=c-1;#
\sk
\verb#q++) {g=x^(L[j,1])*y^(L[j,2]); for (l=1; l<=q;l++) {g=diff(g,x)+U[i]*diff(g,y);}#
\sk
\verb#R[c*(i-1)+q+1,j]=subst(g,x,Q[i,1],y,Q[i,2]);}}} K=rightKernel(R); r=size(K) div N;#
\sk
\verb#q=N-s*c; if (q<0) {q=0;} for(p=1; K[p,1]==0 && p<N; p++){;} if (p==N && K[p,1]==0) #
\sk
\verb#{r=0;} r;if(q==r){C++;W[C]=k+1;}} if(C>0){sprintf("r=max(N-sc,0)=%s at %s",q,W);};#}
\msn
{\it It may be necessary to replace the characters {\small\tt\verb#^,"#} with the ones from a keyboard, especially in some published papers from M.M. and M.Z.}

\end{document}